\documentclass[12pt,reqno]{amsart}
\usepackage{amsmath, amsfonts, amssymb, amsthm, hyperref}
\usepackage{bm}
\allowdisplaybreaks[4]
\def\l{\left}
\def\r{\right}
\def\bg{\bigg}
\def\({\bg(}
\def\){\bg)}
\def\t{\text}
\def\f{\frac}

\def\ord{{\rm ord}}

\def\eq{\equiv}

\def\Z{\mathbb Z}
\def\C{\mathbb C}
\def\N{\mathbb N}

\def\1{{\bf 1}}

\theoremstyle{plain}
\newtheorem{theorem}{Theorem}[section]
\newtheorem{lemma}{Lemma}[section]
\newtheorem{corollary}{Corollary}

\theoremstyle{definition}

\theoremstyle{remark}
\newtheorem{remark}{Remark}

\def\<{\langle}
\def\>{\rangle}
\numberwithin{equation}{section}
\begin{document}
\hbox{}
\medskip

\title[Congruences involving generalized central trinomial coefficients]{Some parametric congruences involving generalized central trinomial coefficients}
\author{Chen Wang}
\address {(Chen Wang) Department of Applied Mathematics, Nanjing Forestry
University, Nanjing 210037, People's Republic of China}
\email{cwang@smail.nju.edu.cn}

\author{Zhi-Wei Sun}
\address {(Zhi-Wei Sun) Department of Mathematics, Nanjing
University, Nanjing 210093, People's Republic of China}
\email{zwsun@nju.edu.cn}

\thanks{This work was supported by the National Natural Science Foundation of China (grant no. 11971222)}
\begin{abstract} For $n\in\N=\{0,1,2,\ldots\}$ and $b,c\in\Z$, the $n$th generalized central trinomial coefficient $T_n(b,c)$ is the coefficient of $x^n$ in the expansion of $(x^2+bx+c)^n$. In particular, $T_n=T_n(1,1)$ is the central trinomial coefficient. In this paper, we mainly establish some parametric congruences involving generalized central trinomial coefficients. As consequences, we prove that for any prime $p>3$
$$
\sum_{k=0}^{p-1}\frac{\binom{2k}{k}}{12^k}T_k\equiv\left(\frac{p}{3}\right)\frac{3^{p-1}+3}{4}\pmod{p^2}
$$
and
$$
\sum_{k=0}^{p-1}\frac{T_kH_k}{3^k}\equiv\frac{3+\left(\frac{p}{3}\right)}{2}-p\left(1+\left(\frac{p}{3}\right)\right)\pmod{p^2},
$$
where $(-)$ denotes the Legendre symbol and $H_k:=\sum_{j=1}^k1/j$ denotes the $k$th harmonic number. These confirm two conjectural congruences of the second author.
\end{abstract}
\maketitle

\section{Introduction}

For any $n\in\N=\{0,1,2,\ldots\}$ and $b,c\in\Z$, the generalized central trinomial coefficient $T_n(b,c)$ introduced by Noe \cite{Noe} is defined as the coefficient of $x^n$ in the expansion of $(x^2+bx+c)^n$ (or the constant term in the expansion of $(x+b+c/x)^n$). By the multinomial theorem, it is clear that
\begin{equation}\label{tnbcformula}
T_n(b,c)=\sum_{k=0}^{\lfloor n/2\rfloor}\binom{n}{2k}\binom{2k}{k}b^{n-2k}c^k,
\end{equation}
where $\lfloor x\rfloor$ denotes the largest integer not exceeding $x$. The generalized central trinomial coefficients have many interesting combinatorial interpretations; for example, from \eqref{tnbcformula}, it is easy to see that $T_n(b,c)$ counts the colored lattice paths from $(0,0)$ to $(n,0)$ using only steps $U=(1,1)$, $D=(1,-1)$ and $H=(1,0)$, where $H$ and $D$ may have $b$ and $c$ colors, respectively. Moreover, $T_n(b,c)$ is a natural extension of some well-known combinatorial sequences; for instance, $T_n:=T_n(1,1)$ is the central trinomial coefficient and $T_n(2,1)$ is exactly the central binomial coefficient $\binom{2n}{n}$. The generalized central trinomial coefficients are also related to the well-known Legendre polynomials
\begin{equation}\label{Legendrepoly}
P_n(x):=\sum_{k=0}^n\binom{n}{k}\binom{n+k}{k}\l(\f{x-1}2\r)^k=\sum_{k=0}^n\f{\binom{n}{k}\binom{2k}{k}}{2^k}(x^2-1)^{k/2}\l(x-\sqrt{x^2-1}\r)^{n-k}
\end{equation}
(cf. \cite[p. 38]{Gould}) via the following identity (see \cite{Noe,Sun2014a,Sun2014b}):
\begin{equation}\label{tntoLegendre}
T_n(b,c)=(\sqrt{d})^n P_n\l(\f{b}{\sqrt{d}}\r),
\end{equation}
where $d=b^2-4c$.

It is known that sums involving products of the binomial coefficients (e.g., $\binom{2k}{k}$, $\binom{2k}{k}^2$, $\binom{2k}{k}\binom{3k}{k}$, $\binom{2k}{k}^3$) usually have some interesting congruence properties. Since $T_n(b,c)$ is a natural extension of $\binom{2n}{n}$, Z.-W. Sun \cite{Sun2014a,Sun2014b} investigated congruences for sums involving the generalized central trinomial coefficients systematically. In particular, Sun \cite[Theorem 2.1]{Sun2014a} determined
$$
\sum_{k=0}^{p-1}\f{\binom{2k}{k}T_k(b,c)}{m^k}\pmod{p}
$$
for any $b,c,m\in\Z$ and odd prime $p$ with $p\nmid m$. As a corollary, he obtained that
\begin{equation}\label{sunres}
\sum_{k=0}^{p-1}\f{\binom{2k}{k}T_k}{12^k}\eq\l(\f{6}{p}\r)\sum_{k=0}^{p-1}\f{\binom{4k}{2k}\binom{2k}{k}}{64^k}\eq\l(\f{p}{3}\r)\pmod{p},
\end{equation}
where $\l(-\r)$ denotes the Legendre symbol. For more congruence properties of the generalized central trinomial coefficients, one may consult \cite{ChenWang, Guo2015,GuoZeng,Liu,Sun2014a,Sun2014b,WangXia2021,Zhangyong}.

For any $n\in\N$ and $x\in\C$, define the polynomial sequence $w_n(x)$ as in \cite{HHT}, i.e.,
$$
w_n(x):=\begin{cases}\f{(\alpha+1)\alpha^n-(\alpha^{-1}+1)\alpha^{-n}}{\alpha-\alpha^{-1}},\quad&\t{if}\ \ x\neq\pm1,\\ 2n+1,\quad&\t{if}\ \ x=1,\\ (-1)^n,\quad&\t{if}\ \ x=-1,\end{cases}
$$
where $\alpha=x+\sqrt{x^2-1}$.

The first purpose of this paper is to establish the following parametric congruence as a generalization of \eqref{sunres}.

\begin{theorem}\label{mainth1}
Let $p$ be an odd prime and let $b,c\in\Z$ with $p\nmid c(b+2c)$. Then
\begin{equation}\label{wangsunexten}
\sum_{k=0}^{p-1}\f{\binom{2k}{k}T_k(b,c^2)}{4^k(b+2c)^k}\eq w_{(p-1)/2}\l(\f{b-6c}{b+2c}\r)\pmod{p^2}.
\end{equation}
\end{theorem}

Taking $(b,c)=(1,1)$ and $(2,1)$, we obtain the following results.
\begin{corollary}\label{cor1}{\rm (i)} For any prime $p>3$, we have
\begin{equation}\label{sunconj}
\sum_{k=0}^{p-1}\f{\binom{2k}{k}T_k}{12^k}\eq\l(\f{p}{3}\r)\f{3^{p-1}+3}{4}\pmod{p^2}.
\end{equation}

{\rm (ii)} For any odd prime $p$, we have
\begin{equation}\label{RVcon}
\sum_{k=0}^{p-1}\f{\binom{2k}{k}^2}{16^k}\eq\l(\f{-1}{p}\r)\pmod{p^2}.
\end{equation}
\end{corollary}

\begin{remark}
{\rm (a)} \eqref{sunconj} confirms a conjecture of Sun \cite[Conjecture 2.1]{Sun2014a}. In 2019, Sun \cite[Conjecture 66]{Sun2019} further conjectured that for any prime $p>3$ and $n\in\Z^{+}$ one has
$$
\f{1}{pn\binom{2n}{n}}\l(\sum_{k=0}^{pn-1}\f{\binom{2k}{k}T_k}{12^k}-\l(\f{p}{3}\r)\sum_{r=0}^{n-1}\f{\binom{2r}{r}T_r}{12^r}\r)\eq \l(\f{p}{3}\r)\f{3^{p-1}-1}{8p}\cdot\f{T_{n-1}}{12^{n-1}}\pmod{p}.
$$
Clearly, when $n=1$, the above conjecture coincides with \eqref{sunconj}. Recently, this conjecture was completely confirmed by Zhang \cite{Zhangyong}.

{\rm(b)} \eqref{RVcon} was conjectured by Rodriguez-Villegas \cite{RVillegas03} and first confirmed by Mortenson \cite{Mortenson03}.
\end{remark}

Our next theorem also concerns sums of the type $\sum_{k=0}^{p-1}\binom{2k}{k}T_k(b,c)/m^k$.

\begin{theorem}\label{mainth2}
Let $p$ be an odd prime and let $b,c\in\Z$ with $p\nmid b$. Then
\begin{equation}\label{mainth2eq1}
\sum_{k=0}^{p-1}\f{\binom{2k}{k}T_k(b,c)}{(4b)^k}\eq p\sum_{k=0}^{(p-1)/2}\f{\binom{2k}{k}}{4k+1}\l(\f{c}{b^2}\r)^k\pmod{p^2}.
\end{equation}
In particular, for any odd prime $p$ we have
\begin{equation}\label{mainth2eq2}
\sum_{k=0}^{p-1}\f{\binom{2k}{k}T_k(2,-1)}{8^k}\eq\begin{cases}2x-p/(2x)\pmod{p^2}\ &if\ p=x^2+4y^2\ \&\ x\eq1\pmod{4},\vspace{2mm}\\
0\pmod{p^2}\ &if\ p\eq3\pmod{4}.\end{cases}
\end{equation}
\end{theorem}

In 2011, Z.-W. Sun \cite[Conjecture 5.5]{Sun2011} conjectured that for any prime $p=x^2+4y^2\eq1\pmod{4}$ with $x\eq1\pmod{4}$ we have
\begin{equation}\label{sunbinomial1}
\sum_{k=0}^{p-1}\f{\binom{2k}{k}^2}{8^k}\eq(-1)^{(p-1)/4}\l(2x-\f{p}{2x}\r)\pmod{p^2}.
\end{equation}
This was later confirmed by Z.-H. Sun \cite{ZHSun}. For any prime $p\eq3\pmod{4}$, Z.-W. Sun \cite{Sun2013} proved that
\begin{equation}\label{sunbinomial2}
\sum_{k=0}^{p-1}\f{\binom{2k}{k}^2}{8^k}\eq\f{(-1)^{(p+1)/4}2p}{\binom{(p+1)/2}{(p+1)/4}}\pmod{p^2}.
\end{equation}

Taking $(b,c)=(2,1)$ in \eqref{mainth2eq1}, and noting \eqref{sunbinomial1} and \eqref{sunbinomial2}, we arrive at the following result.

\begin{corollary}\label{cor2}
Let $p$ be an odd prime. Then
\begin{align*}
&p\sum_{k=0}^{(p-1)/2}\f{\binom{2k}{k}}{(4k+1)4^k}\\
\eq&\begin{cases}(-1)^{(p-1)/4}\l(2x-\f{p}{2x}\r)\pmod{p^2} \ &if\ p=x^2+4y^2\ \&\ x\eq1\pmod{4},\vspace{2mm}\\
(-1)^{(p+1)/4}2p/\binom{(p+1)/2}{(p+1)/4}\pmod{p^2}\ &if\ p\eq3\pmod{4}. \end{cases}
\end{align*}
\end{corollary}

\begin{remark}
Corollary \ref{cor2} can also be directly proved through a similar argument as the one in the proof of \eqref{mainth2eq2} and using the Gauss identity (cf. \cite[p. 66]{AAR99}).
\end{remark}

To state our next theorem, we need to introduce some notations. For $A,B\in\Z$, the Lucas sequences $u_n=u_n(A,B)\ (n\in\N)$ are defined as follows:
$$
u_0=0,\ u_1=1,\ \t{and}\ u_{n+1}=Au_n-Bu_{n-1}\ (n=1,2,3,\ldots).
$$ 
The characteristic equation $x^2-Ax+B=0$ has two roots
$$
\alpha=\f{A+\sqrt{\Delta}}{2}\quad \t{and} \quad \beta=\f{A-\sqrt{\Delta}}{2},
$$
where $\Delta=A^2-4B$. It is known that for any $n\in\N$ we have
$$
u_n=\begin{cases}(\alpha^n-\beta^n)/(\alpha-\beta)\quad&\t{if}\ \Delta\neq0,\\ n(A/2)^{n-1}\quad&\t{if}\ \Delta=0.\end{cases}
$$
It is known (see, e.g., \cite{Sun2012}) that $u_{p-(\f{\Delta}{p})}\eq0\pmod{p}$ for any prime $p\mid 2B$. For $n\in\N$, the $n$th harmonic number is defined by $H_n:=\sum_{k=1}^n1/k$. For any $p$-adic integer $x$, let $\<x\>_{p^r}$ denote the least nonnegative residue of $x$ modulo $p^r$.

\begin{theorem}\label{mainth3}
Let $p$ be an odd prime and let $b,c\in\Z$ with $p\nmid c(b+2c)$. Set $s=(b-2c)/(2c),\ t=\<2s+4\>_{p^2}$ and $d=b^2-4c^2$. Then
\begin{align}\label{mainth3eq1}
&\sum_{k=0}^{p-1}\f{T_k(b,c^2)H_k}{(b+2c)^k}\notag\\
\eq& s+2-s\l(\f{d}{p}\r)(t-4)^{p-1}-2p\l(s+1-s\l(\f{d}{p}\r)-\l(\f{d}{p}\r)^2\f{s(s+1)u_{p-(\f d p)}(t-2,1)}{2p}\r)\notag\\
&+s\begin{cases}(t-4)\pmod{p^2}\quad&if\ p=3\ and\ 3\mid t-1,\\ 0\pmod{p^2}\quad&otherwise.\end{cases}
\end{align}
In particular, for any prime $p>3$ we have
\begin{equation}\label{mainth3eq2}
\sum_{k=0}^{p-1}\f{T_kH_k}{3^k}\eq\f{3+\l(\f{p}{3}\r)}{2}-p\l(1+\l(\f{p}{3}\r)\r)\pmod{p^2}.
\end{equation}
\end{theorem}

\begin{remark}
\eqref{mainth3eq2} was originally conjectured by Z.-W. Sun \cite[Conjecture 1.1 (ii)]{Sun2014b}.
\end{remark}

Putting $b=2,\ c=1$ in \eqref{mainth3eq1}, we obtain the following result concerning central binomial coefficients.

\begin{corollary}
For any odd prime $p$ we have
$$
\sum_{k=0}^{p-1}\f{\binom{2k}{k}H_k}{4^k}\eq 2-2p\pmod{p^2}.
$$
\end{corollary}

The proofs of Theorems \ref{mainth1}--\ref{mainth3} will be given in Sections 2--4, respectively.

\medskip

\section{Proof of Theorem \ref{mainth1}}

In order to show Theorem \ref{mainth1}, we need the following transformation of $T_n(b,c^2)$ which follows from \eqref{tntoLegendre} and \cite[(3.136)]{Gould}.

\begin{lemma}\label{mainth1lem1}
For $n\in\N$ and $b,c\in\Z$ we have
\begin{equation}\label{mainth1lem1eq}
T_n(b,c^2)=\sum_{k=0}^n\binom{n}{k}\binom{2k}{k}(b+2c)^{n-k}(-c)^k.
\end{equation}
\end{lemma}

\begin{proof}
Denote the right-hand side of \eqref{mainth1lem1eq} by $a_n$. Via the Zeilberger algorithm (cf. \cite{PWZ}), we find $T_n(b,c^2)$ and $a_n$ all satisfy the following recurrence relation:
$$
-(b-2c)(b+2c)(n+1)a_n+b(2n+3)a_{n+1}-(n+2)a_{n+2}=0.
$$
Moreover, it is easy to see that $T_0(b,c^2)=a_0$ and $T_1(b,c^2)=a_1$. Then the proof is finished by induction on $n$.
\end{proof}

\begin{lemma}\label{lemma2.1}
Let $n,j\in\N$. Then we have
\begin{equation}\label{lem2.1}
\sum_{k=j}^n\f{\binom{2k}{k}\binom{k}{j}}{4^k}=\f{n+1}{2^{2n+1}(2j+1)}\cdot\binom{n}{j}\binom{2n+2}{n+1}.
\end{equation}
\end{lemma}
\begin{proof}
This could be directly verified by induction on $n$.
\end{proof}

In \cite[Theorem 2]{HHT}, Kh. Hessami Pilehrood, T. Hessami Pilehrood and R. Tauraso completely determined
$$
\sum_{k=0}^{(p-3)/2}\f{\binom{2k}{k}t^k}{2k+1}\pmod{p^3},
$$
where $p$ is an odd prime and $t$ is a $p$-adic unit. In our proof, we only need to use their result in the modulus $p$ case.

\begin{lemma}[cf. {\cite[Theorem 2]{HHT}}]\label{pptres}
For any odd prime $p$ and $t\in\Z_p^{\times}$, we have
\begin{equation}\label{pptreseq}
\sum_{k=0}^{(p-3)/2}\f{\binom{2k}{k}t^k}{2k+1}\eq \f{w_{(p-1)/2}(1-8t)-(-16t)^{(p-1)/2}}{p}\pmod{p}.
\end{equation}
\end{lemma}

\medskip

\noindent{\it Proof of Theorem \ref{mainth1}}. By Lemmas \ref{mainth1lem1} and \ref{lemma2.1}, we have
\begin{align*}
\sum_{k=0}^{p-1}\f{\binom{2k}{k}T_k(b,c^2)}{4^k(b+2c)^k}=&\sum_{k=0}^{p-1}\f{\binom{2k}{k}}{4^k}\sum_{l=0}^k\binom{k}{l}\binom{2l}{l}\l(\f{-c}{b+2c}\r)^l\\
=&\sum_{l=0}^{p-1}\binom{2l}{l}\l(\f{-c}{b+2c}\r)^l\sum_{k=l}^{p-1}\f{\binom{2k}{k}\binom{k}{l}}{4^k}\\
=&\f{p\binom{2p}{p}}{2^{2p-1}}\sum_{l=0}^{p-1}\f{\binom{2l}{l}\binom{p-1}{l}\big(\f{-c}{b+2c}\big)^l}{2l+1}.
\end{align*}
Note that $\binom{2l}{l}/(2l+1)\eq0\pmod{p}$ for $l\in\{(p+1)/2,\ldots,p-1\}$ and $\binom{p-1}{l}\eq (-1)^l\pmod{p}$ for $0\leq l<p$. Thus we have
\begin{equation}\label{mainth1key1}
\sum_{k=0}^{p-1}\f{\binom{2k}{k}T_k(b,c^2)}{4^k(b+2c)^k}\eq \f{\binom{2p}{p}\binom{p-1}{(p-1)/2}^2\big(\f{-c}{b+2c}\big)^{(p-1)/2}}{2^{2p-1}}+\f{p\binom{2p}{p}}{2^{2p-1}}\sum_{l=0}^{(p-3)/2}\f{\binom{2l}{l}\big(\f{c}{b+2c}\big)^l}{2l+1}\pmod{p^2}.
\end{equation}
Wolstenholme \cite{Wolsten} showed that for all primes $p>3$
\begin{equation}\label{wcon}
\binom{2p}{p}\eq2\pmod{p^3}.
\end{equation}
Morley's congruence (cf. \cite{Mo}) states that for any prime $p>3$ we have
\begin{equation}\label{mcon}
\binom{p-1}{(p-1)/2}\eq(-1)^{(p-1)/2}4^{p-1}\pmod{p^3}.
\end{equation}
It can be verified directly that \eqref{wcon} and \eqref{mcon} also hold modulo $p^2$ for $p=3$. Substituting \eqref{wcon} and \eqref{mcon} into \eqref{mainth1key1}, we arrive at
$$
\sum_{k=0}^{p-1}\f{\binom{2k}{k}T_k(b,c^2)}{4^k(b+2c)^k}\eq \l(\f{-16c}{b+2c}\r)^{(p-1)/2}+\f{p}{4^{p-1}}\sum_{l=0}^{(p-3)/2}\f{\binom{2l}{l}\big(\f{c}{b+2c}\big)^l}{2l+1}\pmod{p^2}.
$$
Then we complete the proof by Lemma \ref{pptres} and Fermat's little theorem.\qed

Below we give a short proof of Corollary \ref{cor1}.

\medskip

\noindent{\it Proof of Corollary \ref{cor1}}. \eqref{RVcon} is obvious. To show \eqref{sunconj}, it remains to prove
\begin{equation}\label{cor1key}
w_{(p-1)/2}\l(-\f53\r)\eq\l(\f{p}{3}\r)\f{3^{p-1}+3}{4}\pmod{p^2}.
\end{equation}
It is easy to see that
$$
w_{(p-1)/2}\l(-\f53\r)=\f{\l(-1\r)^{(p-1)/2}}{4}\l((1/3)^{(p-1)/2}+3\times3^{(p-1)/2}\r).
$$
From \cite[p. 51]{IR}, we know that $a^{(p-1)/2}\eq(\f{a}{p})\pmod{p}$ for any integer $a\not\eq0\pmod{p}$. Thus we may write $3^{(p-1)/2}$ as $(\f{3}{p})(1+ph)$, where $h$ is a $p$-adic integer. In view of this,
$$
3^{p-1}=(3^{(p-1)/2})^2\eq1+2p h\pmod{p^2}.
$$
By the above and with the help of the law of quadratic reciprocity (cf. \cite{IR}), we get
\begin{align*}
w_{(p-1)/2}\l(-\f53\r)=&\f{\l(-1\r)^{(p-1)/2}}{4}\l(\f{1}{(\f3p)(1+ph)}+3\l(\f3p\r)(1+ph)\r)\\
\eq&\f{\l(-1\r)^{(p-1)/2}}{4}\l(\f3p\r)(4+2ph)\\
\eq&\f{3^{p-1}+1}{4}\l(\f3p\r)\l(\f{-1}{p}\r)\\
=&\l(\f{p}{3}\r)\f{3^{p-1}+3}{4}\pmod{p^2}.
\end{align*}
This proves \eqref{cor1key}.\qed

\medskip

\section{Proof of Theorem \ref{mainth2}}

Recall that Morita's $p$-adic Gamma function $\Gamma_p$ (cf. \cite{Robert00}) is the $p$-adic analogue of the classical Gamma function $\Gamma$. For $n\in\N$ define $\Gamma_p(0):=1$ and for $n\geq1$
$$
\Gamma_p(n):=(-1)^n\prod_{\substack{1\leq k<n\\ p\nmid k}}k.
$$
Let $\Z_p$ denote the ring of all $p$-adic integers. As we all know, the definition of $\Gamma_p$ can be extended to $\Z_p$ since $\N$ is a dense subset of $\Z_p$ with respect to $p$-adic norm, where $\Z_p$ denotes the ring of $p$-adic integers. It follows that
\begin{equation*}
\f{\Gamma_p(x+1)}{\Gamma_p(x)}=\begin{cases}\displaystyle -x,\quad&\t{if $p\nmid x$},\vspace{2mm}\\ \displaystyle -1,\quad&\t{if $p\mid x$}.\end{cases}
\end{equation*}
It is known (cf. \cite[p. 369]{Robert00}) that for any odd prime $p$ and $x\in\Z_p$ we have
\begin{equation}\label{legendre}
\Gamma_p(x)\Gamma_p(1-x)=(-1)^{\<-x\>_p-1},
\end{equation}
where $\<x\>_p$ is the least nonnegative residue of $x$ modulo $p$. It is also known (cf. \cite{WangPan}) that for any odd prime $p$ and $\alpha,t\in\Z_p$ we have
\begin{equation}\label{padicexpansion}
\Gamma_p(\alpha+tp)\eq\Gamma_p(\alpha)\l(1+tp(\Gamma_p'(0)+H_{p-1-\<-\alpha\>_p})\r)\pmod{p^2}.
\end{equation}
For more properties of $p$-adic gamma functions, one may consult \cite{Robert00}.

The following identity is due to Kummer.
\begin{lemma}[{cf. \cite[p. 126]{AAR99}}]\label{Kummer}
For any $a,b\in\C$ such that the series converges, we have
$$
\sum_{k=0}^{\infty}\f{(a)_k(b)_k(-1)^k}{(a-b+1)_k}=\f{\Gamma(a-b+1)\Gamma(\f a2+1)}{\Gamma(a+1)\Gamma(\f a2-b+1)},
$$
where $(x)_k=x(x+1)\cdots(x+k-1)$ is the Pochhammer symbol.
\end{lemma}

\begin{lemma}\label{keylem} For any odd prime $p$ we have
$$
p\sum_{k=0}^{(p-1)/2}\f{\binom{2k}{k}}{(-4)^k(4k+1)}\eq\begin{cases}\displaystyle2x-\f{p}{2x}\pmod{p^2}\ &if\ p=x^2+4y^2\ \&\ x\eq1\pmod{4},\vspace{2mm}\\
\displaystyle0\ &if\ p\eq3\pmod{4}.\end{cases}
$$
\end{lemma}

\begin{proof}
Clearly,
$$
p\sum_{k=0}^{(p-1)/2}\f{\binom{2k}{k}}{(-4)^k(4k+1)}=p\sum_{k=0}^{(p-1)/2}\f{(\f12)_k(\f14)_k(-1)^k}{(1)_k(\f54)_k}.
$$

We first assume that $p\eq1\pmod{4}$. In this case, $\ord_p(p/(5/4)_k)\geq0$ for all $k$ among $0,1,\ldots,(p-1)/2$. It is easy to check that
\begin{align*}
&p\sum_{k=0}^{(p-1)/2}\f{(\f{1-p}{2})_k(\f{1-2p}{4})_k(-1)^k}{(1)_k(\f54)_k}\\
&\qquad\eq p\sum_{k=0}^{(p-1)/2}\f{(\f12)_k(\f14)_k(-1)^k}{(1)_k(\f54)_k}\l(1-\f{p}2\sum_{j=0}^{k-1}\f{1}{1/2+j}-\f{p}{2}\sum_{j=0}^{k-1}\f{1}{1/4+j}\r)\pmod{p^2}
\end{align*}
and
\begin{align*}
&p\sum_{k=0}^{(p-1)/2}\f{(\f{2-p}{4})_k(\f{1-p}{4})_k(-1)^k}{(1)_k(\f54)_k}\\
&\qquad\eq p\sum_{k=0}^{(p-1)/2}\f{(\f12)_k(\f14)_k(-1)^k}{(1)_k(\f54)_k}\l(1-\f{p}4\sum_{j=0}^{k-1}\f{1}{1/2+j}-\f{p}{4}\sum_{j=0}^{k-1}\f{1}{1/4+j}\r)\pmod{p^2}.
\end{align*}
On the other hand, by Lemma \ref{Kummer} we have
\begin{align*}
&\sum_{k=0}^{(p-1)/2}\f{(\f{1-p}{2})_k(\f{1-2p}{4})_k(-1)^k}{(1)_k(\f54)_k}=\lim_{t\to1}\sum_{k=0}^{\infty}\f{(\f{1-tp}{2})_k(\f{1-2tp}{4})_k(-1)^k}{(1)_k(\f54)_k}\\
=&\lim_{t\to1}\f{\Gamma(\f54)\Gamma(\f{5-tp}4)}{\Gamma(\f{3-tp}2)\Gamma(\f{4+tp}4)}=\f{\Gamma(\f54)\Gamma(\f{p-1}2)}{\Gamma(\f{p-1}4)\Gamma(\f{4+p}4)}\lim_{t\to1}\f{\sin(\f{3-tp}2\pi)}{\sin(\f{5-tp}4\pi)}=(-1)^{(p-1)/4}\f{2\Gamma(\f54)\Gamma(\f{p-1}2)}{\Gamma(\f{p-1}4)\Gamma(\f{4+p}4)},
\end{align*}
where we note that $\Gamma(x)\Gamma(1-x)=\pi/\sin(\pi x)$.
Also,
$$
\sum_{k=0}^{(p-1)/2}\f{(\f{2-p}{4})_k(\f{1-p}{4})_k(-1)^k}{(1)_k(\f54)_k}=\f{\Gamma(\f54)\Gamma(\f{10-p}8)}{\Gamma(\f{6-p}4)\Gamma(\f{8+p}8)}.
$$
Combining the above we obtain
$$
p\sum_{k=0}^{(p-1)/2}\f{(\f12)_k(\f14)_k(-1)^k}{(1)_k(\f54)_k}\eq\sigma_1-\sigma_2\pmod{p^2},
$$
where
$$
\sigma_1:=\f{2p\Gamma(\f54)\Gamma(\f{10-p}8)}{\Gamma(\f{6-p}4)\Gamma(\f{8+p}8)}\quad\t{and}\quad\sigma_2:=(-1)^{(p-1)/4}\f{2p\Gamma(\f54)\Gamma(\f{p-1}2)}{\Gamma(\f{p-1}4)\Gamma(\f{4+p}4)}.
$$
From \cite{L} we know that for any odd prime $p$ we have
$$
H_{\lfloor p/2\rfloor}\eq-2q_p(2)\pmod{p}\quad\t{and}\quad H_{\lfloor p/4\rfloor}\eq-3q_p(2)\pmod{p},
$$
where $q_p(2)=(2^{p-1}-1)/p$ is the Fermat quotient. It is easy to see that $\Gamma((10-p)/8)/\Gamma((8+p)/8)$ contains the factor $8/p$. Thus by \eqref{legendre} and \eqref{padicexpansion} we have
\begin{align*}
\sigma_1=&\f{16\Gamma_p(\f54)\Gamma_p(\f54-\f{p}8)}{\Gamma_p(\f32-\f{p}4)\Gamma_p(1+\f{p}8)}\\
\eq&-\f{16\Gamma_p(\f54)^2}{\Gamma_p(\f32)}\l(1-\f{p}8H_{\lfloor p/4\rfloor}+\f{p}4H_{\lfloor p/2\rfloor}\r)\\
\eq&-2\Gamma_p\l(\f14\r)^2\Gamma_p\l(\f12\r)\l(1-\f{p}8q_p(2)\r)\pmod{p^2}.
\end{align*}
Similarly, it is not hard to find that $\Gamma(5/4)/\Gamma((4+p)/4)$ contains the factor $4/p$. Thus, by \eqref{legendre} and \eqref{padicexpansion} we arrive at
\begin{align*}
\sigma_2=&(-1)^{(p-1)/4}\f{8\Gamma_p(\f54)\Gamma_p(\f{p-1}{2})}{\Gamma_p(\f{p-1}4)\Gamma_p(\f{4+p}4)}\\
\eq&(-1)^{(p+3)/4}\f{8\Gamma_p(\f54)\Gamma_p(-\f12)}{\Gamma_p(-\f14)}\l(1+\f{p}2H_{\lfloor p/2\rfloor}-\f{p}4H_{\lfloor p/4\rfloor}\r)\\
\eq&-\Gamma_p\l(\f14\r)^2\Gamma_p\l(\f12\r)\l(1-\f{p}4q_p(2)\r)\pmod{p^2}.
\end{align*}
In view of the above we have
$$
p\sum_{k=0}^{(p-1)/2}\f{(\f12)_k(\f14)_k(-1)^k}{(1)_k(\f54)_k}\eq-\Gamma_p\l(\f14\r)^2\Gamma_p\l(\f12\r)\pmod{p^2}.
$$
From \cite{CDE}, for $p=x^2+4y^2$ with $x\eq1\pmod{4}$ we have
$$
-\Gamma_p\l(\f14\r)^2\Gamma_p\l(\f12\r)\eq2x-\f{p}{2x}\pmod{p^2}.
$$
This proves the lemma in the case $p\eq1\pmod{4}$.

Below we suppose that $p\eq3\pmod{4}$. In this case, $\ord_p(4j+1)=0$ for all $k$ among $0,1,\ldots,(p-1)/2$. Therefore, by Lemma \ref{Kummer} we have
\begin{align*}
p\sum_{k=0}^{(p-1)/2}\f{(\f12)_k(\f14)_k(-1)^k}{(1)_k(\f54)_k}\eq& p\sum_{k=0}^{(p-1)/2}\f{(\f{1-p}2)_k(\f{1-2p}4)_k(-1)^k}{(1)_k(\f54)_k}\\
=&\f{p\Gamma(\f54)\Gamma(\f{5-p}4)}{\Gamma(\f{3-p}2)\Gamma(\f{4+p}4)}=0\pmod{p^2},
\end{align*}
where in the last step we have used the fact $1/\Gamma(-n)=0$ for any nonnegative integer $n$.

The proof of Lemma \ref{keylem} is now complete.
\end{proof}

\medskip

\noindent{\it Proof of Theorem \ref{mainth2}}. In view of \eqref{tnbcformula} and Lemma \ref{lemma2.1}, we have
\begin{align*}
\sum_{k=0}^{p-1}\f{\binom{2k}{k}T_k(b,c)}{(4b)^k}=&\sum_{k=0}^{p-1}\f{\binom{2k}{k}}{4^k}\sum_{j=0}^{\lfloor k/2\rfloor}\binom{k}{2j}\binom{2j}{j}\l(\f{c}{b^2}\r)^j\\
=&\sum_{j=0}^{(p-1)/2}\binom{2j}{j}\l(\f{c}{b^2}\r)^j\sum_{k=2j}^{p-1}\f{\binom{2k}{k}\binom{k}{2j}}{4^k}\\
=&\f{p\binom{2p}{p}}{2^{2p-1}}\sum_{j=0}^{(p-1)/2}\f{\binom{2j}{j}\binom{p-1}{2j}}{4j+1}\l(\f{c}{b^2}\r)^j.
\end{align*}
If $p\eq3\pmod{4}$, then $p\nmid(4j+1)$ for all $j$ among $0,1,\ldots,(p-1)/2$. In this case, by \eqref{wcon} and Fermat's little theorem, we have
$$
\sum_{k=0}^{p-1}\f{\binom{2k}{k}T_k(b,c)}{(4b)^k}\eq p\sum_{j=0}^{(p-1)/2}\f{\binom{2j}{j}}{4j+1}\l(\f{c}{b^2}\r)^j\pmod{p^2}.
$$
Now suppose $p\eq1\pmod{4}$. Then, by \eqref{wcon} and \eqref{mcon},
\begin{align*}
\sum_{k=0}^{p-1}\f{\binom{2k}{k}T_k(b,c)}{(4b)^k}=&\f{p\binom{2p}{p}}{2^{2p-1}}\sum_{\substack{0\leq j\leq (p-1)/2\\ j\neq (p-1)/4}}\f{\binom{2j}{j}\binom{p-1}{2j}}{4j+1}\l(\f{c}{b^2}\r)^j+\f{\binom{2p}{p}\binom{(p-1)/2}{(p-1)/4}\binom{p-1}{(p-1)/2}}{2^{2p-1}}\l(\f{c}{b^2}\r)^{(p-1)/4}\\
\eq& p\sum_{\substack{0\leq j\leq (p-1)/2\\ j\neq (p-1)/4}}\f{\binom{2j}{j}}{4j+1}\l(\f{c}{b^2}\r)^j+\binom{(p-1)/2}{(p-1)/4}\l(\f{c}{b^2}\r)^{(p-1)/4}\\
=&p\sum_{j=0}^{(p-1)/2}\f{\binom{2j}{j}}{4j+1}\l(\f{c}{b^2}\r)^j\pmod{p^2}.
\end{align*}

Combining the above, we conclude the proof of \eqref{mainth2eq1}.

In particular, for any odd prime $p$ we have
$$
\sum_{k=0}^{p-1}\f{\binom{2k}{k}T_k(2,-1)}{8^k}\eq p\sum_{k=0}^{(p-1)/2}\f{\binom{2k}{k}}{(-4)^k(4k+1)}\pmod{p^2}.
$$
Then \eqref{mainth2eq2} follows from Lemma \ref{keylem} immediately.\qed

\medskip

\section{Proof of Theorem \ref{mainth3}}

The following identity can be verified by induction on $n$.
\begin{lemma}\label{lemma3.1} Let $n,j$ be nonnegative integers. Then we have
\begin{equation}\label{lem3.1}
\sum_{k=j}^n\binom{k}{j}H_k=\binom{n+1}{j+1}\l(H_{n+1}-\f{1}{j+1}\r).
\end{equation}
\end{lemma}

\begin{lemma}\label{mainth3lem2}
For any $n\in\N$, we have
\begin{equation}\label{mainth3lem2eq'}
\sum_{k=0}^n\f{(-(4x+1)k^2+(4nx-2x-2)k+2nx-1)\binom{n}{k}\binom{2k}{k}x^k}{(k+1)^2}=-1,
\end{equation}
or equivalently,
\begin{equation}\label{mainth3lem2eq}
-2(n+1)x\sum_{k=0}^n\f{\binom{n}{k}\binom{2k}{k}x^k}{(k+1)^2}+(4n+6)x\sum_{k=0}^n\f{\binom{n}{k}\binom{2k}{k}x^k}{k+1}-(4x+1)\sum_{k=0}^n\binom{n}{k}\binom{2k}{k}x^k=-1.
\end{equation}
\end{lemma}

\begin{proof}
The identity \eqref{mainth3lem2eq'} is directly obtained by using Zeilberger's algorithm. 
\end{proof}

Taking $a=1$ in \cite[(1.6)]{Sun2010}, we obtain the following result.

\begin{lemma}\label{mainth3lem3}
Let $p$ be an odd prime. Then, for any integer $m\not\eq0\pmod{p}$, we have
\begin{equation}\label{mainth3lem3eq}
\sum_{k=0}^{p-1}\f{\binom{2k}{k}}{(k+1)m^k}\eq 1-\f{m-4}{2}\l(\l(\f{m(m-4)}{p}\r)-1\r)\pmod{p}.
\end{equation}
\end{lemma}

We also need the following result which follows from \cite[(2.7)]{Sun2012} with $h=a=1$.

\begin{lemma}\label{mainth3lem4}
Let $p$ be an odd prime and let $m\in\Z$ with $p\nmid m$. Then
\begin{align}\label{mainth3lem4eq}
\sum_{k=0}^{p-1}\f{\binom{p-1}{k}\binom{2k}{k}}{(-m)^k}\eq& \f{2-m}{2}\l(\f{\Delta}{p}\r)^2u_{p-(\f{\Delta}{p})}(m-2,1)+\l(\f{\Delta}{p}\r)(m-4)^{p-1}\notag\\
&-\begin{cases}(m-4)\pmod{p^2}\quad&if\ p=3\ and\ 3\mid m-1,\\ 0\pmod{p^2}\quad&otherwise.\end{cases}
\end{align}
\end{lemma}

\medskip

\noindent{\it Proof of Theorem \ref{mainth3}}. By Lemmas \ref{mainth1lem1} and \ref{lemma3.1}, we obtain
\begin{align}\label{mainth3key1}
\sum_{k=0}^{p-1}\f{T_k(b,c^2)H_k}{(b+2c)^k}=&\sum_{k=0}^{p-1}H_k\sum_{j=0}^{k}\binom{k}{j}\binom{2j}{j}\l(\f{-c}{b+2c}\r)^j\notag\\
=&\sum_{j=0}^{p-1}\binom{2j}{j}\l(\f{-c}{b+2c}\r)^j\sum_{k=j}^{p-1}\binom{k}{j}H_k\notag\\
=&\sum_{j=0}^{p-1}\binom{2j}{j}\l(\f{-c}{b+2c}\r)^j\binom{p}{j+1}\l(H_{p-1}+\f1p-\f{1}{j+1}\r)\notag\\
\eq&\sum_{j=0}^{p-1}\f{\binom{p-1}{j}\binom{2j}{j}}{j+1}\l(\f{-c}{b+2c}\r)^j-p\sum_{j=0}^{p-1}\f{\binom{p-1}{j}\binom{2j}{j}}{(j+1)^2}\l(\f{-c}{b+2c}\r)^j\pmod{p^2},
\end{align}
where in the last step we have used the fact $H_{p-1}\eq0\pmod{p^2}$. Putting $n=p-1$, $x=-c/(b+2c)$ in \eqref{mainth3lem2eq} we have
\begin{align*}
&\f{-2c}{b+2c}\l(\sum_{j=0}^{p-1}\f{\binom{p-1}{j}\binom{2j}{j}}{j+1}\l(\f{-c}{b+2c}\r)^j-p\sum_{j=0}^{p-1}\f{\binom{p-1}{j}\binom{2j}{j}}{(j+1)^2}\l(\f{-c}{b+2c}\r)^j\r)\\
=&-1+\f{4cp}{b+2c}\sum_{j=0}^{p-1}\f{\binom{p-1}{j}\binom{2j}{j}}{j+1}\l(\f{-c}{b+2c}\r)^j+\f{b-2c}{b+2c}\sum_{j=0}^{p-1}\binom{p-1}{j}\binom{2j}{j}\l(\f{-c}{b+2c}\r)^j.
\end{align*}
Substituting this into \eqref{mainth3key1} we arrive at
\begin{align}\label{mainth3key2}
&\sum_{k=0}^{p-1}\f{T_k(b,c^2)H_k}{(b+2c)^k}\notag\\
\eq&\f{b+2c}{2c}-2p\sum_{j=0}^{p-1}\f{\binom{p-1}{j}\binom{2j}{j}}{j+1}\l(\f{-c}{b+2c}\r)^j-\f{b-2c}{2c}\sum_{j=0}^{p-1}\binom{p-1}{j}\binom{2j}{j}\l(\f{-c}{b+2c}\r)^j\notag\\
\eq&\f{b+2c}{2c}-2p\sum_{j=0}^{p-1}\f{\binom{2j}{j}}{j+1}\l(\f{c}{b+2c}\r)^j-\f{b-2c}{2c}\sum_{j=0}^{p-1}\binom{p-1}{j}\binom{2j}{j}\l(\f{-c}{b+2c}\r)^j\pmod{p^2}.
\end{align}
In view of Lemma \ref{mainth3lem3} with $m=\<(b+2c)/c\>_p$, we have
\begin{equation}\label{mainth3key3}
\sum_{j=0}^{p-1}\f{\binom{2j}{j}}{j+1}\l(\f{c}{b+2c}\r)^j\eq s+1-s\l(\f{d}{p}\r)\pmod{p}.
\end{equation}
Moreover, by Lemma \ref{mainth3lem4} with $m=t=\<(b+2c)/c\>_{p^2}$, we obtain
\begin{align}\label{mainth3key4}
\sum_{j=0}^{p-1}\binom{p-1}{j}\binom{2j}{j}\l(\f{-c}{b+2c}\r)^j\eq&-(s+1)\l(\f{d}{p}\r)^2u_{p-(\f d p)}(t-2,1)+\l(\f{d}{p}\r)(t-4)^{p-1}\notag\\
&-\begin{cases}(t-4)\pmod{p^2}\quad&if\ p=3\ and\ 3\mid t-1,\\ 0\pmod{p^2}\quad&otherwise.\end{cases}
\end{align}
Combining \eqref{mainth3key2}--\eqref{mainth3key4}, we complete the proof of \eqref{mainth3eq1}.

Set $b=c=1$. Then $s=-1/2,\ t=3,\ d=-3$. By \eqref{mainth3eq1}, for $p>3$ we have
$$
\sum_{k=0}^{p-1}\f{T_kH_k}{3^k}\eq \f{3}{2}+\f12\l(\f{-3}{p}\r)-2p\l(\f12+\f12\l(\f{-3}{p}\r)+\f{u_{p-(\f{-3}{p})}(1,1)}{8p}\r)\pmod{p^2}.
$$
Note that by the law of quadratic reciprocity,
$$
\l(\f{-3}{p}\r)\l(\f p3\r)=(-1)^{(p-1)/2}\l(\f{3}{p}\r)\l(\f p3\r)=(-1)^{(p-1)/2}(-1)^{(p-1)/2}=1.
$$
Thus $(\f{-3}{p})=(\f p3)$. It is easy to see that $u_n(1,1)=0$ for $n\in\N$ with $3\mid n$, and
$$
p-\l(\f{p}{3}\r)\eq0\pmod{3}.
$$
Hence $u_{p-(\f{-3}{p})}(1,1)=0$. Combining the above, we arrive at \eqref{mainth3eq2}. \qed

\end{document}